\documentclass[10pt]{article}
\usepackage{amsmath, amsthm, amssymb, amsfonts, mathrsfs, mathtools}
\usepackage{graphicx}
\usepackage{makeidx}
\usepackage[left=2cm,right=2cm,top=2cm,bottom=2cm]{geometry}
\usepackage{caption}
\usepackage{subcaption}
\usepackage{enumitem}
\usepackage{tikz}
\usetikzlibrary{arrows}
\usetikzlibrary{decorations.pathreplacing}
\usepackage{environ}
\usepackage{tikz-cd}
\usepackage{xcolor}
\usepackage[section]{placeins}

\usepackage{tikz}
\usetikzlibrary{decorations.pathreplacing}
\tikzstyle{vertex}=[auto=left,circle,draw=black,fill=white, inner sep=1.5]

\newtheorem{theorem}{Theorem}[section]

\newtheorem{corollary}[theorem]{Corollary}
\newtheorem{lema}[theorem]{Lemma}

\newtheorem{rem}[theorem]{Remark}

\newcommand{\Aut}{\operatorname{Aut}}
\newcommand{\Sp}{\operatorname{Sp}}

\title{Switching equivalence of Hermitian adjacency matrices of mixed graphs}

\author{ Monu Kadyan and Bikash Bhattacharjya\\
Department of Mathematics\\
Indian Institute of Technology Guwahati, India\\
Emails: monu.kadyan@iitg.ac.in, b.bikash@iitg.ac.in}

\date{}

\begin{document}
\maketitle

\vspace{-0.3in}

\begin{center}{\textbf{Abstract}}\end{center}
Let $0 \in \Gamma$ and $\Gamma \setminus \{0\}$ be a subgroup of the  complex numbers of unit modulus.  Define $\mathcal{H}_{n}(\Gamma)$ to be the set of all $n\times n$ Hermitian matrices with entries in $\Gamma$, whose diagonal entries are zero. The matrices $A,B\in \mathcal{H}_{n}(\Gamma)$ are said to be switching equivalent if there is a diagonal matrix $D$, in which the diagonal entries belong to $\Gamma \setminus \{0\}$, such that $D^{-1} A D=B$. We find a characterization, in terms of fundamental cycles of graphs, of switching equivalence of matrices in $\mathcal{H}_{n}(\Gamma)$. We give sufficient conditions to characterize the cospectral matrices in $\mathcal{H}_{n}(\Gamma)$. We find bounds on the number of switching equivalence classes of all mixed graphs with the same underlying graph. We also provide the size of all switching equivalence classes of mixed cycles, and give a formula that calculates the size of a switching equivalence class of a mixed plane graph. We also discuss an action of the automorphism group of a graph on switching equivalence classes of matrices in $\mathcal{H}_{n}(\Gamma)$. 

\vspace*{0.3cm}
\noindent 
\textbf{Keywords:} Mixed graph, Hermitian-adjacency matrix, Switching equivalence\\
\textbf{AMS Classification:} 05C50
\section{Introduction}
 
A  \textit{mixed graph} $G$ is a pair $(V (G), E(G))$, where $V(G)$ is a nonempty set and $E(G)$ is a set of ordered pairs of distinct elements of $V(G)$. The elements of $V(G)$ and $E(G)$ are called the vertices and edges of $G$, respectively. Similarly, the sets $V(G)$ and $E(G)$ are called the vertex set and edge set of $G$, respectively. If $(u,v)\in E(G)$, then the vertices $u$ and $v$ are called the end vertices of the edge $(u,v)$. In a mixed graph $G$, we call an edge with end vertices $u$ and $v$ \textit{undirected} (respectively \textit{directed}) if both $(u,v)$ and $(v,u)$ belong to $E(G)$ (respectively exactly one of $(u,v)$ and $(v,u)$ belongs to $E(G)$). If $(u,v)$ is an undirected edge, then the edges $(u,v)$ and $(v,u)$ are considered identical. Thus, a mixed graph can have both directed and undirected edges. For more information, we refer the reader to \cite{2015mixed}. If all the edges of a mixed graph $G$ are directed (respectively undirected), then $G$ is called a \emph{directed} graph (respectively an \emph{undirected} graph). For a mixed graph $G$, the underlying graph $G_{U}$ of $G$ is the simple undirected graph in which all edges are considered undirected. By the terms order, size, number of components, degree of a vertex, distance between two vertices, walk, path, cycle, tree, bipartiteness etc, we mean the same as in their underlying graphs.

A {\em partial orientation} $\phi$ of an undirected graph $G$, denoted $G^{\phi}$, is obtained by choosing one of the two directed edges in each edge of a subset $S$ of $E(G)$. It is called an {\em orientation} of $G$ if $S=E(G)$, and the resulting graph is called an {\em oriented} or \emph{directed} graph. A partial orientation of $G$ is called trivial if $S=\emptyset $. Thus a mixed graph is obtained from an undirected graph along with a partial orientation $\phi$. So, $M$ is a mixed graph if and only if $M=G^{\phi}$ for some partial orientation $\phi$ of an undirected graph $G$. Define $\mathcal{M}(G)$ to be the set of all mixed graphs having $G$ as the underlying graph, that is,
$$\mathcal{M}(G)=\{G^{\phi} \colon \phi \textnormal{ is a partial orientation of $G$}  \},$$ 
where $G$ is an undirected graph. Note that $\mathcal{M}(G)$ contains $3^m$ distinct mixed graphs with underlying graph $G$ of size $m$.

Liu et al.~\cite{2015mixed}, and also Guo and Mohar~\cite{guo2017hermitian} independently, introduced the Hermitian adjacency matrix of a mixed graph. 
 For a mixed graph $G^{\phi}$ on $n$ vertices, its \textit{Hermitian adjacency matrix} $H(G^{\phi}):=[h_{uv}]_{n\times n}$, is defined by 
 $$h_{uv} = \left\{ \begin{array}{rll}
1 &\mbox{ if } (u,v)\in E \textnormal{ and } (v,u)\in E \\ 
\mathbf{i} & \mbox{ if } (u,v)\in E \textnormal{ and } (v,u)\not\in E \\
-\mathbf{i} & \mbox{ if }  (u,v)\not\in E \textnormal{ and } (v,u)\in E\\
0 &\textnormal{ otherwise.}
\end{array}\right.$$
Here $\mathbf{i}:=\sqrt{-1}$ is the imaginary number unit. The Hermitian-adjacency matrix of a mixed graph can also be obtained as a special case of the adjacency matrix of a 3-colored digraph introduced by Bapat et al. \cite{debajit}. 

Let $\Sp(A)$ denote the multiset of eigenvalues of the matrix $A$. We say that two matrices $A$ and $B$ are cospectral if $\Sp(A)=\Sp(B)$. The spectrum of $G^{\phi}$, denoted $\Sp_H(G^{\phi})$, is the spectrum of $H(G^{\phi})$. It is easy to see that $H(G^{\phi})$ is a Hermitian matrix, and so $\Sp_H(G^{\phi})\subset \mathbb{R}$. We say that two mixed graphs $G^{\phi}$ and $M^{\gamma}$ are cospectral if $H(G^{\phi})$ and $H(M^{\gamma})$ are cospectral.

Let $0 \in \Gamma$ and $\Gamma \setminus \{0\} $ be a subgroup of the complex numbers of unit modulus. Let $\mathcal{H}_{n}(\Gamma)$ be the set of all $n\times n$ Hermitian matrices with entries in $\Gamma$, whose diagonal entries are zero. Let $A:=[a_{ij}]\in \mathcal{H}_{n}(\Gamma)$. Define a graph $G_A$ corresponding to $A$ such that $V(G_A)=\{ 1,\ldots,n\}$ and $E(G_A) = \{ (u,v)\colon  a_{uv}\neq0 \}$. We call $G_A$ the graph of $A$. Note that $G_A$ is an undirected graph. The \emph{gain graph} of $A:=[a_{uv}]\in \mathcal{H}_{n}(\Gamma)$, denoted $G_A^{\zeta}$, is the graph $G_A$ in which an edge $(u,v)$ has the gain $a_{uv}$. The Hermitian adjacency matrix of a mixed graph is a special case of the adjacency matrix of a complex unit gain graph in which the gain set is $ \{1, \mathbf{i}, -\mathbf{i}\}$. For a walk $W:= u_1 \ldots u_k$ in $G_A^{\zeta}$, define the \textit{gain} of $W$ by $\zeta_A(W)\coloneqq a_{u_1u_2} a_{u_2u_3}\ldots a_{u_{k-1}u_k}$. Similarly, for a cycle $C:= u_1 \ldots u_ku_1$ in $G_A^{\zeta}$, the \textit{gain} of $C$ is $\zeta_A(C)\coloneqq a_{u_1u_2} a_{u_2u_3}\ldots a_{u_{k-1}u_k}a_{u_ku_1}$. Let us denote the real part of $\zeta_A(C)$ by $\Re_A(C)$, and call it the \textit{real gain} of $C$ with respect to $A$. A matrix $A \in \mathcal{H}_{n}(\Gamma)$ is said to be \textit{balanced} if the gain of each cycle in $G_A^{\zeta}$ is $1$. 

A switching function of $\Gamma\setminus \{ 0\}$ on $G_A$ is a function $\theta:V(G_A)\to \Gamma\setminus \{ 0\}$. Matrices $A,B\in \mathcal{H}_{n}(\Gamma)$ are said to be switching equivalent, denoted $A\thicksim B$, if there is a switching function $\theta$ of $\Gamma \setminus \{ 0\}$ on $G_A$ such that  $D(\theta)^{-1} A D(\theta)=B$, where $D(\theta)= \text{diag}[ \theta(1),\ldots,\theta(n)]$. Switching equivalence on $\mathcal{H}_{n}(\Gamma)$ was also discussed in \cite{reff, zaslavsky1989}. Let $\overline{z}$ denote the conjugate of a complex number $z$. The condition $D(\theta)^{-1} A D(\theta)=B$, where $D(\theta)=\text{diag}[ \theta(1),\ldots,\theta(n)]$, implies that an edge $(u,v)$ has the gain $a_{uv}$ in $G_A^{\zeta}$ if and only if it has the gain $\overline{\theta(u)}\theta(v)a_{uv}$ in $G_B^{\zeta}$. It is clear  that if $A\thicksim B$, then $G_A=G_B$.

Let $D(\theta)^{-1} A D(\theta)=B$, where $D(\theta)=\text{diag}[ \theta(1),\ldots,\theta(n)]$ and $\theta(i)\in \Gamma\setminus \{ 0\}$ for each $i\in \{1,\ldots,n\}$. Let $C: v_1v_2\ldots v_kv_1$ be a cycle in $G_A$. We have
\begin{equation*}
\begin{split}
\zeta_{A}(C)&=a_{v_1v_2}a_{v_2v_3}\cdots a_{v_kv_1}\\
&= \overline{\theta(v_1)}a_{v_1v_2}\theta(v_2)\overline{\theta(v_2)}a_{v_2v_3}\theta(v_3)\cdots \overline{\theta(v_k)}a_{v_kv_1}\theta(v_1)\\
&=b_{v_1v_2}b_{v_2v_3}\cdots b_{v_kv_1}\\
&=\zeta_{B}(C).
\end{split}
\end{equation*}
Thus, if $A\thicksim B$ then  $\zeta_A(C)=\zeta_B(C)$ for each cycle $C$ in $G_A$. If $A\thicksim B$, then we also say that the gain graphs $G_A^{\zeta}$ and $G_B^{\zeta}$ are equivalent, and it is denoted by $G_A^{\zeta}\thicksim G_B^{\zeta}$.

The existence of cospectral graphs is one of the widespread problems in graph theory. The seminal paper by Collatz and Sinogowitz \cite{von1957spektren} provided the first example of cospectral trees. Cospectral graphs are studied for signed graphs \cite{akbari2018signed}, oriented graphs \cite{denglan2013skew}, mixed graphs \cite{mohar2016hermitian}, complex unit gain graphs \cite{samanta2019spectrum} etc. In all these cases, to characterize the cospectral graphs, researchers introduce a switching equivalence relation, which is a similarity transformation. It is clear that if two matrices in $\mathcal{H}_{n}(\Gamma)$ are switching equivalent then they are similar, and hence they are cospectral. In 1982, Thomas Zaslavsky \cite{zaslavsky1982signed} characterized switching equivalence on signed graphs in terms of balance of cycles. In 2020, Yi Wang and Bo-Jun Yuan  \cite{wang2020graphs} characterized switching equivalence on mixed graphs using a strong cycle basis. Note that the cycle basis with respect to a maximal forest of a graph is indeed a strong cycle basis. Thus, strong cycle bases are not so important in connection with gains on mixed graphs. Guo and Mohar \cite{guo2017hermitian} presented an operation on mixed graphs, called four-way switching, which can be described as switching equivalence relation that preserves the spectrum of mixed graphs. The authors in \cite{debajit} and \cite{reff} simultaneously and independently defined complex unit gain graphs, their adjacency matrices, and switching. However, the terminologies of \cite{debajit} and \cite{reff} are different.

The paper is organized as follows. In the second section, we introduce a characterization of switching equivalence on $\mathcal{H}_{n}(\Gamma)$. A sufficient condition is also presented that characterizes a bipartite graph $G_A$ in terms of the eigenvalues of a matrix $A$. We also provide sufficient conditions to characterize cospectral matrices in $\mathcal{H}_{n}(\Gamma)$. In the third section, we find bounds on the number of the switching equivalence classes of all mixed graphs with the same underlying graph. In the fourth section, we find the size of all switching equivalence classes of mixed cycles, and give a formula that calculates the size of a switching equivalence class of a mixed plane graph. In the last section,  we discuss an action of the automorphism group of a graph on switching equivalence classes of matrices in $\mathcal{H}_{n}(\Gamma)$.

\section{Switching equivalence in $\mathcal{H}_{n}(\Gamma)$ }

There are various types of adjacency matrices in algebraic graph theory, which are defined on graphs. For example, the adjacency matrix of a signed graph, the Hermitian adjacency matrix of a complex unit gain graph, and the Hermitian adjacency matrix of a mixed graph. In all these cases, the notion of switching equivalence is defined. This section defines switching equivalence on Hermitian matrices with zero diagonal entries and characterizes them using fundamental cycles. This, in turn, characterizes switching equivalence in a signed graph, mixed graph, and complex unit gain graph.

Let $G$ be a graph and $F$ be a maximal forest of $G$. Define $$\mathcal{B}_F(G)\coloneqq \{ C_F(e):e\in E(G)\setminus E(F)\},$$ where $C_F(e)$ denotes the unique cycle in $F\cup \{e\}$. The cycle $C_F(e)$ is called the \textit{fundamental cycle} of $G$ with respect to $F$ and $e$, and $\mathcal{B}_F(G)$ is called the \textit{fundamental cycle basis} of $G$ with respect to $F$.  We say that the matrices $A,B\in \mathcal{H}_{n}(\Gamma)$ have the same graph if $G_A=G_B$.  

\begin{lema}\label{1}
Let $A\in \mathcal{H}_{n}(\Gamma)$ and $F$ be a maximal forest of $G_A$. Then there is a matrix $ [a^\prime_{ij}]\in \mathcal{H}_{n}(\Gamma)$ such that $A\thicksim [a^\prime_{ij}]$, where $a^\prime_{ij}=1$ for each edge $(i,j)$ in $F$. Further, the switching is unique up to a scalar multiple in each component of $F$.
\end{lema}
\begin{proof}Let $A\in \mathcal{H}_{n}(\Gamma)$ and $F$ be a maximal forest of $G_A$. Let a root be chosen in each of the components of $F$. For each $w\in V(G)$, define
\[\theta(w)= \left\{ \begin{array}{cl}
		1 &\mbox{ if } w=v \\
		\zeta_A(P_{wv}) &\textnormal{ otherwise,}
	\end{array}\right.\]
where $v$ is the root of the component of $F$ containing $w$, and $P_{wv}$ is the unique $wv$-path in $F$. Let
\[D(\theta)\coloneqq \text{diag}[ \theta(1),\ldots,\theta(n)]\text{ and } [a_{ij}^\prime] \coloneqq D(\theta)^{-1}AD(\theta).\]
It is clear that $[a_{ij}^\prime]\in \mathcal{H}_{n}(\Gamma)$ and $A\thicksim [a_{ij}^\prime]$. If $(i,j)\in E(F)$ and $v$ is the root of the component of $F$ containing $i$, then we have
\begin{equation*}
	\begin{split}
			a_{ij}^\prime &= \overline{\theta(i)}a_{ij}\theta(j) \\
				            &= \left\{ \begin{array}{ll}
					               \overline{\zeta_A(P_{iv})}a_{ij} a_{ji}\zeta_A(P_{iv}) & \textnormal{ if } (j,i) \textnormal{ appears on } P_{jv} \\ 
					               \overline{a_{ij}\zeta_A(P_{jv})} a_{ij}\zeta_A(P_{jv}) & \textnormal{ if } (j,i) \textnormal{ does not appear on } P_{jv}
				               \end{array}\right.\\
				            &= 1.
	\end{split} 
\end{equation*}
This proves the first part of the lemma. 	Now suppose $ [b^\prime_{ij}]\in \mathcal{H}_{n}(\Gamma)$ such that $A\thicksim [b^\prime_{ij}]$, where $b^\prime_{ij}=1$ for each edge $(i,j)$ in $F$. Then there exists $\pi:V(G_A)\to \Gamma\setminus \{ 0\}$ such that $[b^\prime_{ij}]=D(\pi)^{-1}BD(\pi)$, where $D(\pi)= \text{diag}[ \pi(1),\ldots,\pi(n)]$. Let the number of components of $F$ be $k$, and let $v_1,\ldots,v_k$ be the roots (chosen) in the components of $F$. For convenience, assume $\pi(v_r)=\alpha_r$ for each $r\in \{1,\ldots,k\}$. Let $w$ be a vertex in the component of $F$ having $v_r$ as a root, and let $P_{wv_r}\coloneqq w w_1w_2\cdots w_t v_r$ be the unique $wv_r$-path in $F$. We have
\begin{align*}
  & b^\prime_{ww_1}b^\prime_{w_1w_2}\cdots b^\prime_{w_tv_r}=1\\
\text{or, }  & \overline{\pi(w)}a_{ww_1}\pi(w_1) \overline{\pi(w_1)}a_{w_1w_2}\pi(w_2)\cdots \overline{\pi(w_t)}a_{w_tv_r}\pi(v_r)=1\\
\text{or, }  & \pi(w)= \pi(v_r)a_{ww_1}a_{w_1w_2}\cdots a_{w_tv_r} =\alpha_r \zeta_A(P_{wv_r})=\alpha_r \theta(w) .
\end{align*}
Hence the switching $\pi$ is unique up to a scalar multiple in each component of $F$.
\end{proof}

\begin{theorem}\label{sweq}
If $A,B\in \mathcal{H}_{n}(\Gamma)$ have the same graph $G$, then $A$ and $B$ are switching equivalent if and only if $\zeta_{A}(C)=\zeta_{B}(C)$ for each cycle $C\in \mathcal{B}_F(G)$, where $F$ is a maximal forest of $G$.
\end{theorem}
\begin{proof} Let $F$ be a maximal forest of $G$. Apply Lemma~\ref{1} to find matrices $A^\prime:=[a_{ij}^\prime]$ and $B^\prime:=[b_{ij}^\prime]$ such that $A\thicksim A^\prime, B\thicksim B^\prime$ and $a_{ij}^\prime=1=b_{ij}^\prime$ for each edge $(i,j)$ in $F$. Assume that  $\zeta_{A}(C)=\zeta_{B}(C)$ for all cycles $C\in \mathcal{B}_F(G)$. Now $A^\prime=D(\theta)^{-1}AD(\theta)$  and  $B^\prime=D(\pi)^{-1}BD(\pi)$ for some  $D(\theta)=\text{diag}[ \theta(1),\ldots,\theta(n)]$ and $D(\pi)=\text{diag}[ \pi(1),\ldots,\pi(n)]$, where $\theta(i), \pi(i)\in \Gamma\setminus \{ 0\}$ for each $i\in \{1,\ldots,n\}$. This gives that  $\zeta_{A^\prime}(C)=\zeta_{B^\prime}(C)$ for each cycle $C\in \mathcal{B}_F(G)$. 

Let $(i,j)\notin E(F)$, and let $C$ be the cycle in $\mathcal{B}_F(G)$ containing the edge $(i,j)$. We have
\[a_{ij}^\prime=\zeta_{A^\prime}(C)=\zeta_{B^\prime}(C)=b_{ij}^\prime.\]
Thus $A^\prime=B^\prime$. This gives $D(\theta)^{-1}AD(\theta)=D(\pi)^{-1}BD(\pi) $, and hence $A$ and $B$ are switching equivalent. 

Conversely, assume that $A$ and $B$ are switching equivalent. Then clearly $\zeta_{A}(C)=\zeta_{B}(C)$ for each cycle $C\in \mathcal{B}_F(G)$.
\end{proof}

An easy consequence of Theorem~\ref{sweq} on mixed graphs is presented in the following. 

\begin{corollary}\label{weight} 
If $A,B\in \mathcal{H}_{n}(\Gamma)$ have the same graph $G$, then $\zeta_{A}(C)=\zeta_{B}(C)$ for each cycle $C\in \mathcal{B}_F(G)$ if and only if $\zeta_{A}(C)=\zeta_{B}(C)$ for each cycle $C$ in $G$. 
\end{corollary}
\begin{proof}
Converse part is trivial. Assume that $\zeta_{A}(C)=\zeta_{B}(C)$ for each cycle $C\in \mathcal{B}_F(G)$. Then Theorem ~\ref{sweq} gives that $A$ and $B$ are switching equivalent. Hence $\zeta_{A}(C)=\zeta_{B}(C)$ for each cycle $C$ in $G$. 
\end{proof}

The following corollaries are immediate consequences of Theorem~\ref{sweq}.

\begin{corollary} Let $A,B\in \mathcal{H}_{n}(\Gamma)$ be two matrices with the same graph $G$. If $a_{uv}=b_{uv}$ for each non cut-edge $(u,v)$ in $G$, then $A$ and $B$ are switching equivalent.
\end{corollary}

\begin{corollary} Let $A,B\in \mathcal{H}_{n}(\Gamma)$ be two matrices with the same graph $G$. If $G$ is a forest, then $A$ and $B$ are switching equivalent.
\end{corollary}

\begin{theorem}
If $A,B\in \mathcal{H}_{n}(\Gamma)$ have the same graph $G$, then $A$ and $B$ are switching equivalent if and only if $\zeta_{A}(C)=\zeta_{B}(C)$ for each chordless cycle $C$ in $G$.
\end{theorem}
\begin{proof}
 Assume that $\zeta_{A}(C)=\zeta_{B}(C)$ for each chordless cycle $C$ in $G$. We now show that $\zeta_{A}(C)=\zeta_{B}(C)$ for each cycle $C$ in $G$. Suppose, on the contrary, that there is a cycle $C$ in $G$ of least length such that $\zeta_{A}(C)\neq \zeta_{B}(C)$. So, $C$ contains a chord, say $v_1u_1$ such that $C:=v_1\ldots v_pu_1\ldots u_qv_1$. Let $C_1 \coloneqq v_1\ldots v_pu_1v_1$ and $C_2 \coloneqq u_1\ldots u_qv_1u_1$. If $C_1$ and $C_2$ are chordless, then clearly $\zeta_{A}(C_j)=\zeta_{B}(C_j)$ for $j\in \{1,2\}$. If either $C_1$ or $C_2$ has a chord, then also by the choice of $C$, we have $\zeta_{A}(C_j)=\zeta_{B}(C_j)$ for $j\in \{1,2\}$. Since $a_{u_1v_1}.\overline{a}_{u_1v_1}=1=b_{u_1v_1}.\overline{b}_{u_1v_1}$, we have $$\zeta_{A}(C)=\frac{\zeta_{A}(C_1) {\zeta_{A}(C_2)}}{a_{u_1v_1}\overline{a}_{u_1v_1}}=\frac{\zeta_{B}(C_1) {\zeta_{B}(C_2)}}{b_{u_1v_1}\overline{b}_{u_1v_1}}= \zeta_{B}(C).$$ This contradicts our assumption. Thus the result follows. The proof of the other part follows directly from the equivalence of $A$ and $B$. 
\end{proof}

In the next theorem, we characterize switching equivalence of the matrices $A$ and $-A$ in $\mathcal{H}_{n}(\Gamma)$ in terms of bipartite graphs.
\begin{theorem} If $A\in \mathcal{H}_{n}(\Gamma)$, then $A$ and $-A$ are switching equivalent if and only if $G_A$ is bipartite.
\end{theorem}
\begin{proof} The matrices $A$ and $-A$ are switching equivalent if and only if $\zeta_{-A}(C)=\zeta_{A}(C)$ for all cycles $C$ in $G_A$, which is true if and only if the length of each cycle in $G_A$ is even. Hence the proof follows.
\end{proof}

\begin{corollary}\label{EigSymZero} Let $A\in \mathcal{H}_{n}(\Gamma)$. If $G_A$ is bipartite, then the eigenvalues of $A$ are symmetric about zero.
\end{corollary}
\begin{proof} Switching equivalence of $A$ and $-A$ implies that eigenvalues of $A$ are symmetric about zero.
\end{proof}

The converse of Corollary ~\ref{EigSymZero} need not be true. Consider the Hermitian adjacency matrix of the mixed graph $C_3^{\phi}$ in Figure \ref{figure1}. Here spectrum of $H(C_3^{\phi})$ is $\{ 0,\pm \sqrt{3}\}$. However, $C_3$ is not bipartite.

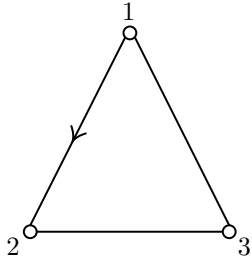
\begin{figure}[ht]
\centering
\tikzset{every picture/.style={line width=0.75pt}} 
\begin{tikzpicture}[x=0.25pt,y=0.25pt,yscale=-1,xscale=1]
\draw    (400.25,391.08) -- (544.29,105.75) ;
\draw    (558.71,106.75) -- (701.25,391.08) ;
\draw    (408.96,400.79) -- (690.54,400.79) ;
\draw  [color={rgb, 255:red, 0; green, 0; blue, 0 }  ,draw opacity=1 ][line width=0.75]  (691.54,400.79) .. controls (691.54,395.43) and (695.89,391.08) .. (701.25,391.08) .. controls (706.61,391.08) and (710.96,395.43) .. (710.96,400.79) .. controls (710.96,406.15) and (706.61,410.5) .. (701.25,410.5) .. controls (695.89,410.5) and (691.54,406.15) .. (691.54,400.79) -- cycle ;
\draw  [color={rgb, 255:red, 0; green, 0; blue, 0 }  ,draw opacity=1 ][line width=0.75]  (541.29,99.75) .. controls (541.29,94.39) and (545.64,90.04) .. (551,90.04) .. controls (556.36,90.04) and (560.71,94.39) .. (560.71,99.75) .. controls (560.71,105.11) and (556.36,109.46) .. (551,109.46) .. controls (545.64,109.46) and (541.29,105.11) .. (541.29,99.75) -- cycle ;
\draw  [line width=0.75]  (485.35,251.61) .. controls (475.76,254.32) and (468.78,258.36) .. (464.39,263.72) .. controls (466.18,257.02) and (465.38,248.99) .. (461.99,239.63) ;
\draw  [color={rgb, 255:red, 0; green, 0; blue, 0 }  ,draw opacity=1 ][line width=0.75]  (390.54,400.79) .. controls (390.54,395.43) and (394.89,391.08) .. (400.25,391.08) .. controls (405.61,391.08) and (409.96,395.43) .. (409.96,400.79) .. controls (409.96,406.15) and (405.61,410.5) .. (400.25,410.5) .. controls (394.89,410.5) and (390.54,406.15) .. (390.54,400.79) -- cycle ;
\draw (535,50) node [anchor=north west][inner sep=0.75pt]    {$1$};
\draw (360,405) node [anchor=north west][inner sep=0.75pt]    {$2$};
\draw (709.96,405) node [anchor=north west][inner sep=0.75pt]    {$3$};
\end{tikzpicture}
\caption{Mixed graph $C_3^{\phi}$}\label{figure1}
\end{figure}

Assume that $A\in \mathcal{H}_{n}(\Gamma)$. An \textit{elementary} graph with respect to $A$ is a subgraph of $G_A$ such that every component is an edge or a cycle. Let $\mathcal{E}_k(A)$ denote the set of all elementary graphs with respect to $A$ of order $k$. For any $H\in \mathcal{E}_k(A)$, let $\mathcal{C}(H)$ denote the set of cycles in $H$, $c(H)$ denote the number of cycles in $H$, and $k(H)$ denote the number of components of $H$. Let the characteristic polynomial of $A$ be
$$\Phi(A,x)\coloneqq x^n + a_1x^{n-1} +\cdots + a_n.$$ 
Recall that $\Re_A(C)$ denotes the real part of $\zeta_A(C)$. The following results appear in~\cite{germina2012balance}.

\begin{theorem}[\cite{germina2012balance}]\label{CharPol}  
The coefficients of the characteristic polynomial of a matrix $A\in \mathcal{H}_{n}(\Gamma)$ are given by
$$ a_k= \sum_{H\in \mathcal{E}_k(A)} (-1)^{k(H)} 2^{c(H)} \prod_{C \in \mathcal{C}(H)} \Re_A(C) \textnormal{ for all } k\in \{1,\ldots,n\}.$$
\end{theorem}

\begin{corollary}[\cite{germina2012balance}] If $A\in \mathcal{H}_{n}(\Gamma)$, then $$\mathrm{det}(A)=(-1)^n \sum_{H\in \mathcal{E}_n(A)} (-1)^{k(H)} 2^{c(H)} \prod_{C \in \mathcal{C}(H)} \Re_A(C).$$
\end{corollary}

In the discussion following Corollary~\ref{EigSymZero}, we observed that even if the eigenvalues of a matrix $A$ in $\mathcal{H}_{n}(\Gamma)$ are symmetric about zero, the corresponding graph $G_A$ need not be bipartite. However, the converse of Corollary~\ref{EigSymZero} is true under a sufficient condition, as mentioned in the next theorem. 

\begin{theorem}[\cite{mehatari2020adjacency}]
 Let $A\in \mathcal{H}_{n}(\Gamma)$ and $\sum_{C\in \mathcal{C}_{k}(G_A)} \Re_A(C)\neq 0$ for all $k$, where $\mathcal{C}_k(G_A)$ is the set of all cycles of length $k$ in $G_A$. If eigenvalues of $A$ are symmetric about zero, then $G_A$ is bipartite.
\end{theorem}

Let $A,B\in \mathcal{H}_{n}(\Gamma)$ have the same graph $G$. Using Theorem~\ref{CharPol}, we observe that if $\Re_A(C)= \Re_B(C)$ for all cycles $C$ in $G$, then $A$ and $B$ are cospectral. However, the converse of this statement need not be true. For example, consider the Hermitian adjacency matrices of the mixed graphs $G^{\phi}$ and $G^{\gamma}$ in Figure~\ref{figure2}. Here both $H(G^{\phi})$ and $H(G^{\gamma})$ are cospectral with the spectrum $\{ 0, \pm 1, \pm \sqrt{5}\}$. However, $\zeta_{H(G^{\phi})}(C)=-1$ and $\zeta_{H(G^{\gamma})}(C)=\mathbf{i}$ for the cycle $C:=123$.

\begin{figure}[ht]
\centering
\begin{subfigure}{0.45\textwidth}
\tikzset{every picture/.style={line width=0.75pt}} 
\begin{tikzpicture}[x=0.25pt,y=0.25pt,yscale=-1,xscale=1]
\draw    (110.71,100.75) -- (391.54,250.79) ;
\draw    (100,110.46) -- (99.25,391.08) ;
\draw    (109.96,398.79) -- (390.54,250.79) ;
\draw  [color={rgb, 255:red, 0; green, 0; blue, 0 }  ,draw opacity=1 ][line width=0.75]  (90.54,400.79) .. controls (90.54,395.43) and (94.89,391.08) .. (100.25,391.08) .. controls (105.61,391.08) and (109.96,395.43) .. (109.96,400.79) .. controls (109.96,406.15) and (105.61,410.5) .. (100.25,410.5) .. controls (94.89,410.5) and (90.54,406.15) .. (90.54,400.79) -- cycle ;
\draw  [color={rgb, 255:red, 0; green, 0; blue, 0 }  ,draw opacity=1 ][line width=0.75]  (91.29,100.75) .. controls (91.29,95.39) and (95.64,91.04) .. (101,91.04) .. controls (106.36,91.04) and (110.71,95.39) .. (110.71,100.75) .. controls (110.71,106.11) and (106.36,110.46) .. (101,110.46) .. controls (95.64,110.46) and (91.29,106.11) .. (91.29,100.75) -- cycle ;
\draw  [line width=0.75]  (253.29,162.29) .. controls (255.81,171.93) and (259.71,178.99) .. (264.98,183.49) .. controls (258.32,181.56) and (250.28,182.21) .. (240.85,185.41) ;
\draw  [color={rgb, 255:red, 0; green, 0; blue, 0 }  ,draw opacity=1 ][line width=0.75]  (390.54,250.79) .. controls (390.54,245.43) and (394.89,241.08) .. (400.25,241.08) .. controls (405.61,241.08) and (409.96,245.43) .. (409.96,250.79) .. controls (409.96,256.15) and (405.61,260.5) .. (400.25,260.5) .. controls (394.89,260.5) and (390.54,256.15) .. (390.54,250.79) -- cycle ;
\draw  [color={rgb, 255:red, 0; green, 0; blue, 0 }  ,draw opacity=1 ][line width=0.75]  (691.54,100.79) .. controls (691.54,95.43) and (695.89,91.08) .. (701.25,91.08) .. controls (706.61,91.08) and (710.96,95.43) .. (710.96,100.79) .. controls (710.96,106.15) and (706.61,110.5) .. (701.25,110.5) .. controls (695.89,110.5) and (691.54,106.15) .. (691.54,100.79) -- cycle ;
\draw  [color={rgb, 255:red, 0; green, 0; blue, 0 }  ,draw opacity=1 ][line width=0.75]  (690.54,400.79) .. controls (690.54,395.43) and (694.89,391.08) .. (700.25,391.08) .. controls (705.61,391.08) and (709.96,395.43) .. (709.96,400.79) .. controls (709.96,406.15) and (705.61,410.5) .. (700.25,410.5) .. controls (694.89,410.5) and (690.54,406.15) .. (690.54,400.79) -- cycle ;
\draw    (409.96,250.79) -- (691.54,100.79) ;
\draw    (409.96,250.79) -- (690.54,400.79) ;
\draw    (700.25,391.08) -- (701.25,110.5) ;
\draw  [line width=0.75]  (265.5,330.71) .. controls (255.93,327.94) and (247.87,327.67) .. (241.3,329.9) .. controls (246.37,325.16) and (249.93,317.92) .. (252.01,308.19) ;
\draw  [line width=0.75]  (540.17,166.25) .. controls (549.45,169.89) and (557.45,170.88) .. (564.19,169.26) .. controls (558.72,173.51) and (554.51,180.4) .. (551.56,189.91) ;
\draw  [line width=0.75]  (687.75,259.64) .. controls (695.15,252.97) and (699.62,246.26) .. (701.18,239.5) .. controls (702.54,246.3) and (706.81,253.15) .. (714,260.04) ;
\draw (85,50) node [anchor=north west][inner sep=0.75pt]    {$1$};
\draw (85,417) node [anchor=north west][inner sep=0.75pt]    {$3$};
\draw (685,413) node [anchor=north west][inner sep=0.75pt]    {$5$};
\draw (385,200) node [anchor=north west][inner sep=0.75pt]    {$2$};
\draw (685,50) node [anchor=north west][inner sep=0.75pt]    {$4$};
\draw (394,382.4) node [anchor=north west][inner sep=0.75pt]    {$G^{\phi}$};
\end{tikzpicture}
\end{subfigure}
\hfill
\begin{subfigure}{0.45\textwidth}
\tikzset{every picture/.style={line width=0.75pt}} 
\begin{tikzpicture}[x=0.25pt,y=0.25pt,yscale=-1,xscale=1]
\draw    (110.71,100.75) -- (391.54,250.79) ;
\draw    (100,110.46) -- (99.25,391.08) ;
\draw    (109.96,398.79) -- (390.54,250.79) ;
\draw  [color={rgb, 255:red, 0; green, 0; blue, 0 }  ,draw opacity=1 ][line width=0.75]  (90.54,400.79) .. controls (90.54,395.43) and (94.89,391.08) .. (100.25,391.08) .. controls (105.61,391.08) and (109.96,395.43) .. (109.96,400.79) .. controls (109.96,406.15) and (105.61,410.5) .. (100.25,410.5) .. controls (94.89,410.5) and (90.54,406.15) .. (90.54,400.79) -- cycle ;
\draw  [color={rgb, 255:red, 0; green, 0; blue, 0 }  ,draw opacity=1 ][line width=0.75]  (91.29,100.75) .. controls (91.29,95.39) and (95.64,91.04) .. (101,91.04) .. controls (106.36,91.04) and (110.71,95.39) .. (110.71,100.75) .. controls (110.71,106.11) and (106.36,110.46) .. (101,110.46) .. controls (95.64,110.46) and (91.29,106.11) .. (91.29,100.75) -- cycle ;
\draw  [line width=0.75]  (253.29,162.29) .. controls (255.81,171.93) and (259.71,178.99) .. (264.98,183.49) .. controls (258.32,181.56) and (250.28,182.21) .. (240.85,185.41) ;
\draw  [color={rgb, 255:red, 0; green, 0; blue, 0 }  ,draw opacity=1 ][line width=0.75]  (390.54,250.79) .. controls (390.54,245.43) and (394.89,241.08) .. (400.25,241.08) .. controls (405.61,241.08) and (409.96,245.43) .. (409.96,250.79) .. controls (409.96,256.15) and (405.61,260.5) .. (400.25,260.5) .. controls (394.89,260.5) and (390.54,256.15) .. (390.54,250.79) -- cycle ;
\draw  [color={rgb, 255:red, 0; green, 0; blue, 0 }  ,draw opacity=1 ][line width=0.75]  (691.54,100.79) .. controls (691.54,95.43) and (695.89,91.08) .. (701.25,91.08) .. controls (706.61,91.08) and (710.96,95.43) .. (710.96,100.79) .. controls (710.96,106.15) and (706.61,110.5) .. (701.25,110.5) .. controls (695.89,110.5) and (691.54,106.15) .. (691.54,100.79) -- cycle ;
\draw  [color={rgb, 255:red, 0; green, 0; blue, 0 }  ,draw opacity=1 ][line width=0.75]  (690.54,400.79) .. controls (690.54,395.43) and (694.89,391.08) .. (700.25,391.08) .. controls (705.61,391.08) and (709.96,395.43) .. (709.96,400.79) .. controls (709.96,406.15) and (705.61,410.5) .. (700.25,410.5) .. controls (694.89,410.5) and (690.54,406.15) .. (690.54,400.79) -- cycle ;
\draw    (409.96,250.79) -- (691.54,100.79) ;
\draw    (409.96,250.79) -- (690.54,400.79) ;
\draw    (700.25,391.08) -- (701.25,110.5) ;
\draw  [line width=0.75]  (540.17,166.25) .. controls (549.45,169.89) and (557.45,170.88) .. (564.19,169.26) .. controls (558.72,173.51) and (554.51,180.4) .. (551.56,189.91) ;
\draw (85,50) node [anchor=north west][inner sep=0.75pt]    {$1$};
\draw (85,417) node [anchor=north west][inner sep=0.75pt]    {$3$};
\draw (685,413) node [anchor=north west][inner sep=0.75pt]    {$5$};
\draw (385,200) node [anchor=north west][inner sep=0.75pt]    {$2$};
\draw (685,50) node [anchor=north west][inner sep=0.75pt]    {$4$};
\draw (394,382.4) node [anchor=north west][inner sep=0.75pt]    {$G^{\gamma}$};
\end{tikzpicture}
\end{subfigure}
\caption{The graphs $G^{\phi}$ and $G^{\gamma}$} \label{figure2}
\end{figure}
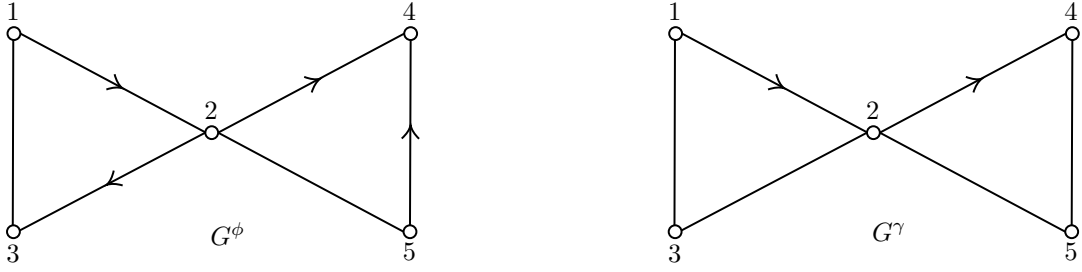

 In the next two results, we establish two sufficient conditions under which $\Re_A(C)= \Re_B(C)$ for each cycle $C$ in $G$ whenever $A$ and $B$ are two cospectral matrices in $\mathcal{H}_{n}(\Gamma)$ with the same graph $G$. Some similar results appear in \cite{samanta2019spectrum}. Our proof technique is also similar with that in \cite{samanta2019spectrum}.

\begin{theorem}\label{SameSpec}
Let $A,B\in \mathcal{H}_{n}(\Gamma)$ have the same graph $G$. If $\Re_A(C)\leq \Re_B(C)$ for each cycle $C$ in $G$, then $A$ and $B$ are cospectral if and only if $\Re_A(C)= \Re_B(C)$ for each cycle $C$ in $G$.
\end{theorem}
\begin{proof} Let $\Phi(A,x) = x^n + a_1x^{n-1} +\cdots + a_n$ and $\Phi(B,x) = x^n + b_1x^{n-1} +\cdots + b_n$ be the characteristic polynomials of $A$ and $B$, respectively. Assume that $\Re_A(C)\leq \Re_B(C)$ for each cycle $C$ in $G$, and that $\Sp(A)=\Sp(B)$. Then we have $a_k=b_k$ for each $k\in \{1,\ldots,n\}$. We apply induction on the length of cycles. Let $\mathcal{C}_k(G)$ be the set of all cycles of length $k$ in $G$. Note that $a_3=b_3$, and therefore we have  $-2 \sum_{C\in \mathcal{C}_{3}(G)} \Re_A(C)=-2 \sum_{C\in \mathcal{C}_{3}(G)} \Re_B(C)$. The given condition $\Re_A(C)\leq \Re_B(C)$ implies that $\Re_A(C)= \Re_B(C)$ for each $C\in  \mathcal{C}_{3}(G)$. Assume that the statement holds for each cycle of length less than or equal to $l$. By induction hypothesis, we have

\begin{equation}\label{eq1samespec}
\begin{split}
\sum_{H\in \mathcal{E}_{l+1}(A) \setminus \mathcal{C}_{l+1}(G)} (-1)^{k(H)} 2^{c(H)} \prod_{C \in \mathcal{C}(H)} \Re_A(C)= \sum_{H\in \mathcal{E}_{l+1}(B)\setminus \mathcal{C}_{l+1}(G)} (-1)^{k(H)} 2^{c(H)} \prod_{C \in \mathcal{C}(H)} \Re_B(C).
\end{split}
\end{equation}
Since $a_{l+1}=b_{l+1}$, we have 
\begin{equation}\label{eq2samespec}
\begin{split}
\sum_{H\in \mathcal{E}_{l+1}(A)} (-1)^{k(H)} 2^{c(H)} \prod_{C \in \mathcal{C}(H)} \Re_A(C)= \sum_{H\in \mathcal{E}_{l+1}(B)} (-1)^{k(H)} 2^{c(H)} \prod_{C \in \mathcal{C}(H)} \Re_B(C).
\end{split}
\end{equation}
From (\ref{eq1samespec}) and (\ref{eq2samespec}), we get $$-2 \sum_{C\in \mathcal{C}_{l+1}(G)} \Re_A(C)=-2 \sum_{C\in \mathcal{C}_{l+1}(G)} \Re_B(C).$$ 

Now $\Re_A(C)\leq \Re_B(C)$ implies that $\Re_A(C)= \Re_B(C)$ for each $C\in  \mathcal{C}_{l+1}(G)$. Hence by induction, $\Re_A(C)= \Re_B(C)$ for each cycle $C$ in $G$. The other part of the theorem is Theorem~\ref{CharPol}.
\end{proof}

\begin{corollary}\label{SameSpecCoro}
Let $A,B\in \mathcal{H}_{n}(\Gamma)$ have the same graph $G$. If $A$ is balanced, then $A$ and $B$ are cospectral if and only if $\Re_A(C)= \Re_B(C)$ for each cycle $C$ in $G$.
\end{corollary}
\begin{proof} Since  $\Re_B(C)\leq 1 = \Re_A(C)$ for each cycle $C$ in $G$, by Theorem~\ref{SameSpec} we get the desired result. 
\end{proof}

Some characterizations of balance in gain graphs (Theorem 2.4 of ~\cite{germina2012balance}) and signed graphs~\cite{acharya1980spectral} are consequences of Corollary~\ref{SameSpecCoro}. In the next section, we introduce a similar characterization of balanced mixed graphs. The proof of the following result is similar to the proof of Theorem \ref{SameSpec}.

\begin{theorem}\label{SameSpec2}
Let $A,B\in \mathcal{H}_{n}(\Gamma)$ have the same graph $G$. Suppose that $\Re_A(C_1)=\Re_A(C_2)$ and $\Re_B(C_1)=\Re_B(C_2)$ for every pair of cycles $C_1$ and $C_2$ of the same length in $G$. Then $A$ and $B$ are cospectral if and only if $\Re_A(C)= \Re_B(C)$ for each cycle $C$ in $G$.
\end{theorem}

\section{Switching equivalence on mixed graphs}
All definitions for complex unit gain graphs apply to mixed graphs with the gains in $\{1,-1,\mathbf{i},-\mathbf{i}\}$. Thus, the terms \emph{switching equivalence} of mixed graphs, \emph{gain} (respectively \emph{real gain}) of walks in a mixed graph, \emph{balanced} mixed graphs etc. have the same meaning as in the corresponding Hermitian adjacency matrices. A mixed graph $G^{\phi}$ is \textit{negative} (respectively \textit{imaginary}) if the \textit{gain} with respect to $G^{\phi}$ of each cycle in $G$ is $-1$ (respectively $\mathbf{i}$ or $-\mathbf{i}$).

Switching equivalence of the mixed graphs $G^{\phi}$ and $G^{\gamma}$ is denoted by $G^{\phi}\thicksim G^{\gamma}$. Switching equivalence is an equivalence relation on $\mathcal{M}(G)$. An equivalence class under this switching equivalence is the  \textit{switching equivalence class} of a mixed graph. The switching equivalence class containing the mixed graph  $ G^{\phi}$ is denoted by $[G^{\phi}]$.  The set of all switching equivalence classes of a mixed graph $G$ is denoted by $\Omega_{\mathcal{M}}(G)$. It is straight forward to see that if two mixed graphs $G^{\phi}$ and $G^{\gamma}$ are switching equivalent, then $G^{\phi}$ and $G^{\gamma}$ are cospectral. 

In $2020$,  Yi Wang and  Bo-Jun Yuan \cite{wang2020graphs} introduced a strong cycle basis for mixed graphs and proved that every graph admits a strong cycle basis. Using a strong cycle basis, they proved that two mixed graphs are switching equivalent if and only if the gains of each cycle in the strong cycle basis are same. Indeed, the cycle basis with respect to a maximal forest is always a strong cycle basis. Thus, the characterization of switching equivalent mixed graphs in \cite{wang2020graphs} is a special case of Theorem ~\ref{sweq}, where  $\Gamma= \{ 0,\pm 1,\pm \mathbf{i}\}$. We present this in the next result.

\begin{theorem}\label{sw01}
Let $\phi$ and $\gamma$ be two partial orientations of a graph $G$. If $F$ is a maximal forest of $G$, then $G^{\phi}$ and $G^{\gamma}$ are switching equivalent if and only if $\zeta_{G^{\phi}}(C)=\zeta_{G^{\gamma}}(C)$ for each cycle $C \in \mathcal{B}_F(G)$.
\end{theorem} 

Using  Corollary~\ref{SameSpecCoro}, we now present a characterization of balanced mixed graphs.
\begin{theorem}\label{3.2} Let $\phi$ be a partial orientation of a graph $G$. Then $G^{\phi}$ is balanced if and only if $G$ and $G^{\phi}$ are cospectral.
\end{theorem}
\begin{proof}
Let $A=H(G)$ and $B=H(G^{\phi})$. Note that $G$ is undirected, and so $A$ is always balanced. Now by Corollary~\ref{SameSpecCoro}, $A$ and $B$ are cospectral if and only if $\Re_A(C)= \Re_B(C)$ for each cycle $C$ in $G$. As $A$ is balanced, $\Re_{A}(C)=1$ for each cycle $C$ in $G$. Further, $\zeta_{B}(C)\in \{\pm 1, \pm \mathbf{i}\}$, and therefore $\Re_B(C)=1$ gives that $\zeta_{B}(C)=1$. Thus $A$ and $B$ are cospectral if and only if $\zeta_{B}(C)=1$ for each cycle $C$ in $G$, that is, $G$ and $G^{\phi}$ are cospectral if and only if $G^{\phi}$ is balanced.
\end{proof}
We note that Theorem~\ref{3.2} for gains in $\mathbb{R} \setminus \{0\}$ is Theorem 1 of \cite{acharya1980spectral}. Given a graph $G$ on $n$ vertices and $m$ edges, there are $3^m$ ways of constructing a mixed graph on $G$. We give a natural lower bound and upper bound of $|\Omega_{\mathcal{M}}(G)|$ in the next result.

\begin{theorem}\label{sw06}
Let $G$ be a graph of order $n$, size $m$ and $c$ components. Then there are at least  $3^{m-n+c}$ and at most $4^{m-n+c}$ distinct mixed graphs up to switching equivalence on $G$. Further, equality occurs in the upper bound if each fundamental cycle of $G$ with respect to a maximal forest $F$ has at least two edges that do not lie on other fundamental cycles with respect to $F$.
\end{theorem}
\begin{proof}
Let $F$ be a maximal forest of $G$. Consider each edge of $F$ to be undirected. Each $e\in E(G)\setminus E(F)$ lies on exactly one fundamental cycle $C_F(e)$. So we can get the gains $1,\pm \mathbf{i}$ on $C_F(e)$ by assigning a suitable direction on the edge $e$. We can do this with all fundamental cycles in $\mathcal{B}_F(G)$.  So there are three possible choices of gains on each fundamental cycle in $\mathcal{B}_F(G)$.  Hence we find at least $3^{m-n+c}$ distinct mixed graphs up to switching equivalence on $G$. Further, there are at most four gains, namely, $ \pm 1, \pm \mathbf{i}$, possible for a fundamental cycle. Hence there are at most $4^{m-n+c}$ distinct mixed graphs up to switching equivalence on $G$.

Let $e,f$ be two edges of a cycle $C$ in $\mathcal{B}_F(G)$ that do not lie on other fundamental cycles in $\mathcal{B}_F(G)$. It is clear that by assigning appropriate gains to $e$ and $f$, and leaving the edges in $C-\{e,f\}$ undirected, we can get the gain of $C$ to be any given number in $\{\pm 1, \pm \mathbf{i}\}$. Hence,  equality occurs in the upper bound in this case.
\end{proof}

A graph $G$ is a \emph{cactus} graph if every block of $G$ is a cycle or a single edge.

\begin{theorem}\label{sw07}
If $G^{\phi}$ is a negative or imaginary mixed graph, then each component of $G$ is a cactus graph.
\end{theorem}
\begin{proof}
Suppose, for a contradiction, that $G$ is not a cactus graph. Then there are two cycles $C_1$ and $C_2$ in $G$ such that $E(C_1)\cap E(C_2)\neq \emptyset$. Let $P:= x_1\ldots x_k$, $C_1:= x_1v_1\ldots v_px_kx_{k-1}\ldots x_1$ and $C_2:= x_1\ldots x_ku_1\ldots u_qx_1$. Let $C:=x_1v_1\ldots v_px_ku_1\ldots u_qx_1$ satisfy the condition that $E(C)=E(C_1) \Delta E(C_2)$, where  $\Delta$ is the symmetric difference of sets. It is clear that $$\zeta_{G^{\phi}}(C)=\zeta_{G^{\phi}}(C_1)\zeta_{G^{\phi}}(P)\overline{\zeta_{G^{\phi}}(P)}\zeta_{G^{\phi}}(C_2)=\zeta_{G^{\phi}}(C_1)\zeta_{G^{\phi}}(C_2).$$
If $G^{\phi}$ is a negative mixed graph, then $\zeta_{G^{\phi}}(C_1) =\zeta_{G^{\phi}}(C_2)=-1$. This gives $\zeta_{G^{\phi}}(C)=1$, which is a contradiction. If $G^{\phi}$ is an imaginary mixed graph, then $\zeta_{G^{\phi}}(C_1), \zeta_{G^{\phi}}(C_2) \in \{ \mathbf{i},-\mathbf{i} \}$. This gives that  $\zeta_{G^{\phi}}(C)\in \{ 1, -1\}$, a contradiction again. Hence $G$ must be a a cactus graph.
\end{proof} 

Let $G^{\phi}$ and $D^{\gamma}$ be two mixed graphs. The cartesian product $G^{\phi} \square D^{\gamma}$ of $G^{\phi}$ and $D^{\gamma}$ is a mixed graph with vertex set $V (G) \times V (D)$ and there is an undirected edge (respectively directed edge) from  $(g_1,d_1)$ to $(g_2,d_2)$ if and only if $g_1 = g_2$ and $(d_1,d_2)$ is an undirected edge (respectively directed edge) in $D^{\gamma}$, or $d_1 = d_2$ and $(g_1, g_2)$ is an undirected edge (respectively directed edge) in $G^{\phi}$. One can see that  $$H(G^{\phi} \square D^{\gamma})=I_{|V(G)|} \otimes H(D^{\gamma})+H(G^{\phi}) \otimes I_{|V(D)|}.$$

\begin{theorem}
If $G^{\phi_1} \thicksim G^{\gamma_1}$ and $D^{\phi_2} \thicksim D^{\gamma_2}$, then $G^{\phi_1} \square D^{\phi_2} \thicksim G^{\gamma_1} \square D^{\gamma_2}$.
\end{theorem}
\begin{proof}
Since $G^{\phi_1} \thicksim G^{\gamma_1}$ and $D^{\phi_2} \thicksim D^{\gamma_2}$, there exist diagonal matrices $D(\theta_1)$ and $D(\theta_2)$ such that $H(G^{\phi_1})=D(\theta_1)^{-1}H(G^{\gamma_1})D(\theta_1)$ and  $H(D^{\phi_2})=D(\theta_2)^{-1}H(D^{\gamma_2})D(\theta_2)$. Thus $$(D(\theta_1) \otimes D(\theta_2))^{-1}H(G^{\gamma_1} \square D^{\gamma_2})   (D(\theta_1) \otimes D(\theta_2))=  H(G^{\phi_1} \square D^{\phi_2}).  $$
Hence $G^{\phi_1} \square D^{\phi_2} \thicksim G^{\gamma_1} \square D^{\gamma_2}$.
\end{proof}


\section{Size of switching equivalence classes of mixed graphs}

Let $G^{\phi}$ be a mixed graph. Recall that $[G^{\phi}] =\{ G^{\gamma} \colon G^{\phi} \textnormal{ and  }G^{\gamma} \textnormal{ are switching equivalent}\}$.

\begin{theorem}\label{sw09}
Let $G$ be a graph and $S$ be the collection of all cut-edges of $G$. Then the size of each switching equivalence class of a mixed graph over $G$ is at least $3^{|S|}$. If $G$ is a forest, then the equality occurs.
\end{theorem}
\begin{proof}
Let $G^{\phi}$ be a mixed graph over $G$. By Theorem~\ref{sw01}, the switching equivalence relation depends on the gain of fundamental cycles. Further, the edges of $S$ do not belong to any cycle. So, changing the gains on the edges of $S$, the graph remains switching equivalent to $G^{\phi}$. Since there are only three gains possible on each edge, the size of $[G^{\phi}]$ is at least $3^{|S|}$.  If $G$ is a forest then $|E(G)|=|S|$ and $G^{\phi}$ is always balanced. Hence there is only one switching equivalence class, and its size is $3^{|S|}$.
\end{proof}

Let $\{ 1,\pm \mathbf{i}\}^n:=\{(z_1,\ldots,z_n)\in \mathbb{C}^n\colon  z_j\in \{  1,\pm \mathbf{i}\} \textnormal{ for each } j \}$. For $z\in \{ 1,\pm \mathbf{i}\}^n$, define $\zeta(z):=z_1\cdots z_n$, that is, $\zeta(z)$ is the product of the coordinates of $z$. Let $\alpha(0)= [ 1 ~~ 0 ~~ 0~~ 0 ]^t$. For $z\in \{ 1,\pm \mathbf{i}\}^n$ and $n\geq 1$, let $\alpha(n):= [ \alpha_1(n)~~ \alpha_{-1}(n) ~~ \alpha_{\mathbf{i}}(n)~~\alpha_{-\mathbf{i}}(n) ]^t$ be the vector in $\mathbb{R}^4$ such that $\alpha_x(n)$ is the number of $z$ in $\{ 1,\pm \mathbf{i}\}^n$ with $\zeta(z)=x$, where $x\in \{ \pm1,\pm \mathbf{i}\}$. Thus $\alpha(1)= [ 1 ~~ 0 ~~ 1~~ 1 ]^t$.  In general, $\alpha_x(n)=|[C_n^{\phi}]|$, where $x=\zeta_{C_n^{\phi}}(C_n)$. Let $I_n$ be the $n\times n$ identity matrix, and let $J_n$ be the $n\times n$ matrix with all entries $1$. Further, let 
$$L= \begin{bmatrix} I_2 & J_{2} \\ J_{2}& I_2 \end{bmatrix}.$$

\begin{lema}\label{sw21}
If $n\geq 1$, then $\alpha(n)=L \alpha(n-1)$.
\end{lema}
\begin{proof}
Use induction on $n$. As $\alpha(1)= [ 1 ~~ 0 ~~ 1~~ 1 ]^t$, we have $\alpha(1)=L \alpha(0)$. Assume that the statement is true for $n\leq k$. Now take $n=k+1$.

\noindent Claim: $\alpha_1(k+1)=\alpha_1(k)+\alpha_{\mathbf{i}}(k)+\alpha_{-\mathbf{i}}(k)$, that is, $\alpha_1(k+1) = [ 1 ~~ 0 ~~ 1~~ 1 ] \alpha(k)$. Each  $z\in \{ 1,\pm \mathbf{i}\}^{k+1}$ can be written as $z=(\tilde{z},z_{k+1})$, where $\tilde{z} \in \{ 1,\pm \mathbf{i}\}^{k}$ and $z_{k+1}\in \{ 1,\pm \mathbf{i}\}$. If $\zeta(z)=1$, then 
$$z_{k+1}=1 \Rightarrow \zeta(\tilde{z})=1,~~z_{k+1}=\mathbf{i} \Rightarrow \zeta(\tilde{z})=-\mathbf{i} ~\text{ and }~ z_{k+1}=-\mathbf{i} \Rightarrow \zeta({\tilde{z}})=\mathbf{i}.$$
Hence $\alpha_1(k+1)=\alpha_1(k)+\alpha_{\mathbf{i}}(k)+\alpha_{-\mathbf{i}}(k)$, that is, $\alpha_{1}(k+1) = [ 1 ~~ 0 ~~ 1~~ 1 ] \alpha(k)$. Similarly, one can show that $\alpha_{-1}(k+1) = [ 0 ~~ 1 ~~ 1~~ 1 ] \alpha(k)$, $\alpha_{\mathbf{i}}(k+1) = [ 1 ~~ 1 ~~ 1~~ 0 ] \alpha(k)$  and $\alpha_{-\mathbf{i}}(k+1) = [ 1 ~~ 1 ~~ 0~~ 1 ] \alpha(k)$. Thus $\alpha(k+1)=L \alpha(k)$, and the proof follows by induction.
\end{proof}

\begin{lema}\label{sw22}
If  $Q=J_4-L$, then
$$L^{n}= \left\{ \begin{array}{rcl}
\frac{1}{4}(3^{n}+1)J_4-Q 
& \textnormal{if $n$ is odd} \\ \frac{1}{4}(3^{n}-1)J_4+I_4  &  \textnormal{if $n$ is even.}
\end{array}\right. $$
\end{lema}
\begin{proof}
Since $L=J_4-Q$, we have
\begin{equation*}
\begin{split}
L^{n}&= (J_4-Q)^{n}\\
&= \sum_{j=0}^{n}    \binom{n}{j}  J_{4}^{n-j} (-1)^j Q^j , \textnormal{ as $J_4$ and $Q$ commute}\\
&=  \sum_{j=0}^{n-1}    \binom{n}{j}  4^{n-1-j}  (-1)^j J_4 Q^j + (-1)^{n} Q^{n}, \textnormal{ as } J_4^{n-j}=4^{n-1-j} J_4\textnormal{ for all } j\\
&=\frac{1}{4}\bigg( \sum_{j=0}^{n-1}    \binom{n}{j}  4^{n-j} (-1)^j \bigg) J_4+ (-1)^{n} Q^{n}, \textnormal{ as } J_4Q^j=J_4\\
&=\frac{1}{4}\bigg( 3^{n} -(-1)^{n} \bigg) J_4+ (-1)^{n} Q^{n}.
\end{split}
\end{equation*}
Now using the facts that  if $n$ is even then $Q^{n}=I$, and if $n$ is odd then $Q^{n}=Q$, we get the desired conclusion.
\end{proof}

\begin{theorem} If $n\geq 3$ and $C_n^{\phi}$ is a mixed cycle of order $n$, then
\[|[C_n^{\phi}]| = \left\{ \begin{array}{ll}
\frac{3^n+1}{4} & \textnormal{if $n$ is odd and }\zeta_{C_n^{\phi}}(C_n)\in \{1,\pm\mathbf{i}\}\\
\frac{3^n-3}{4} & \textnormal{if $n$ is odd and }\zeta_{C_n^{\phi}}(C_n)=-1\\
\frac{3^n+3}{4} & \textnormal{if $n$ is even and }\zeta_{C_n^{\phi}}(C_n)=1\\
\frac{3^n-1}{4} & \textnormal{if $n$ is even and }\zeta_{C_n^{\phi}}(C_n)\in \{-1,\pm\mathbf{i}\}.
\end{array}\right.\]
\end{theorem}
\begin{proof}
From the definition of $\alpha_x(n)$, we know that $|[C_n^{\phi}]|=\alpha_x(n)$, where $x=\zeta_{C_n^{\phi}}(C_n)\in \{\pm 1,\pm \mathbf{i}\}$. Lemma ~\ref{sw21} gives that $\alpha(n)=L^{n}\alpha(0)$, where $\alpha(0)=[ 1 ~~ 0 ~~ 0~~ 0 ]^t$. Therefore
\begin{equation*}
\begin{split}
[ \alpha_1(n)~ \alpha_{-1}(n) ~ \alpha_{\mathbf{i}}(n)~\alpha_{-\mathbf{i}}(n) ]^t=\text{the first column of }L^n =  \left\{\begin{array}{rcl}
\left[ \frac{3^n+1}{4} ~~ \frac{3^n-3}{4} ~~ \frac{3^n+1}{4}~~ \frac{3^n+1}{4} \right]^t&  \textnormal{if $n$ is odd,}\\
\left[ \frac{3^n+3}{4} ~~  \frac{3^n-1}{4}  ~~ \frac{3^n-1}{4}~~\frac{3^n-1}{4} \right]^t
& \textnormal{if $n$ is even.}
\end{array}\right.
\end{split}
\end{equation*} 
Thus the result follows.
\end{proof}

Let $R$ be a finite subset of $\Gamma\setminus \{0\}$. A \textit{restricted gain graph} $G$ with gain set $R$, denoted $G^R$, is a gain graph in which the gain of each edge belongs to the set $R$. Switching of a restricted gain graph is done by the elements in $\Gamma\setminus \{0\}$, with the added condition that the switched gain graph is also restricted in $R$. Note that mixed graphs with the Hermitian adjacency matrix are restricted gain graphs, where $R = \{1, \mathbf{i}, -\mathbf{i}\}$ in the group $\{1, -1, \mathbf{i}, -\mathbf{i}\}$. It is clear that the switching of restricted gain graphs is also an equivalence relation. The equivalence class, under this equivalence relation, containing the restricted gain graph $G^R$ is denoted by $[G^R]$. 

\begin{rem}\label{r1} \normalfont
By Theorem~\ref{sweq}, two restricted gain graphs, with the same underlying graph and same gain set, are switching equivalent if and only if they have the same gain on each cycle of a fundamental cycle basis of the underlying graph.  
\end{rem}

\begin{theorem}\label{size1}
Let $G^R$ be a restricted gain graph with a finite gain set $R$ and underlying graph $G$. If $G_1,\ldots, G_k$ are the blocks of $G$, then 
\[|[G^R]|=\prod_{j=1}^{k}|[G_j^R]|,\]
where $G_j^R$ is the restricted gain graph with underlying graph $G_j$ in which gains are the same as that in $G^R$ for each $j\in \{1,\ldots,k\}$.
\end{theorem}
\begin{proof} Let $\tilde{G}^R$ be another restricted gain graph with the finite gain set $R$ and underlying graph $G$. Also, let $\tilde{G}_j^R$ be the restricted gain graph with underlying graph $G_j$ in which gains are the same as that in $\tilde{G}^R$ for each $j\in \{1,\ldots,k\}$.

Let $F$ be a spanning subgraph of $G$ and $F_j=F\cap G_j$ for each $j\in \{1,\ldots,k\}$. Clearly, $F$ is a maximal forest of $G$ if and only if $F_j$ is a spanning tree of $G_j$ for each $j\in \{1,\ldots,k\}$. Thus $\mathcal{B}_F(G)=\cup_{j=1}^{k}\mathcal{B}_{F_j}(G_j)$, a disjoint union. This fact, along with Remark~\ref{r1}, gives that $G^R$ and $\tilde{G}^R$ are switching equivalent if and only if $G^R_j$ and $\tilde{G}^R_j$ are switching equivalent for each $j\in \{1,\ldots,k\}$. Hence 
\[  |[G^R]|=\prod_{j=1}^{k}|[G_j^R]|. \qedhere \]
\end{proof}

\begin{corollary}
Let $G^\phi$ be a mixed graph such that $G$ is a cactus graph. If $\{ C_{n_1},\ldots,C_{n_k}\} $ is the set of cycles in $G$, then 
$$|[G^{\phi}]|=3^{|S|} \prod_{j=1}^{k} \alpha_{x_j}(n_j),$$ 
where $S$ is the set of all cut-edges of $G$ and 
$x_j=\zeta_{G^{\phi}}(C_{n_j})$ for each $j\in \{1,\ldots,k\}$.
\end{corollary}
\begin{proof} 
Note that the blocks of $G$ are the edges of $S$ and the cycles $C_{n_1},\ldots,C_{n_k}$. If $x_j=\zeta_{G^\phi}(C_{n_j})$, then by definition, $|[C_{n_j}^{\phi}]|=\alpha_{x_j}(n_j)$ for each $j\in \{1,\ldots,k\}$. Also, if the complete graph $K_2$ represents an edge in $S$, then $|[K_2^{\phi}]|=3$. Thus by Theorem~\ref{size1}, we have
\[ |[G^{\phi}]|=3^{|S|} \prod_{j=1}^{k} \alpha_{x_j}(n_j). \qedhere \]
\end{proof}

Now our aim is to calculate the size $|[G^{\phi}]|$ of a mixed plane graph $G^{\phi}$. Due to Theorem~\ref{size1}, it is enough to consider $G$ to be $2$-connected. Let $G$ be a $2$-connected plane graph, and let $f$ be an inner (bounded) face of $G$. The boundary of $f$ may be regarded as a subgraph. We use the notation $\partial (f)$ to denote the edge set of this subgraph. A cycle $C$ in a plane graph $G$ is said to be a \textit{face cycle} if $E(C)= \partial (f)$ for some inner face $f$. Note that a face cycle is defined only for the inner faces of a plane graph. It is known that the face cycles of a $2$-connected plane graph form a basis of the cycle space of the graph. See \cite{diestel} for details. 

Let $G$ be a $2$-connected mixed plane graph and $C_{n_1},\ldots,C_{n_k}$ be the distinct face cycles in $G$, each considered in clockwise direction. Let  $n_r$ be the order of $C_{n_r}$ for each $r\in \{1,\ldots,k\}$.  For $p\neq q$, define $E_{pq}=E(C_{n_p})\cap E(C_{n_q}) $ and $E_{p/q}=E(C_{n_p})\setminus E(C_{n_q}) $. Also, let $E_{pp}$ be the set of all edges that lie only on the cycle $C_{n_p}$. Note that $E_{pp}$ does not contain edges of $C_{n_q}$ for $q\neq p$. It is clear that $\bigcup_{p,q=1, p\leq q}^{k}E_{pq}$ is a disjoint union of all the edges of $G$.  Let $n_{pq} = |E_{pq}|$ for $p,q\in \{1,\ldots,k\}$.

Let $\phi$ be a partial orientation on $G$. Let $\zeta_{G^{\phi}}(E_{pq})$ be the product of gains, with respect to $G^{\phi}$, of the edges in $E_{pq}$ according to the clockwise direction of $C_{n_p}$. Similarly, $\zeta_{G^{\phi}}(E_{p/q})$ is also defined. Note that $\overline{\zeta_{G^{\phi}}(E_{pq})}=\zeta_{G^{\phi}}(E_{qp})$. If $C_{n_p} \Delta C_{n_q}$ is a cycle, then we have
\begin{align}\label{symdif}
\zeta_{G^{\phi}}(C_{n_p} \Delta C_{n_q}) = &\zeta_{G^{\phi}}(E_{p/q})\zeta_{G^{\phi}}(E_{q/p})\\ \nonumber
=& \zeta_{G^{\phi}}(E_{p/q})\zeta_{G^{\phi}}(E_{pq})\zeta_{G^{\phi}}(E_{qp}) \zeta_{G^{\phi}}(E_{q/p})  \\ \nonumber
=& \zeta_{G^{\phi}}(C_{n_p})\zeta_{G^{\phi}}(C_{n_q}). \nonumber
\end{align}
Let $\mathcal{C}=\{C_{n_1},\ldots,C_{n_k}\}$. Note that $\mathcal{C}$ is a basis of the cycle space of $G$. Therefore, if $C$ is a fundamental cycle of $G$, then $C$ can be expressed as a symmetric difference of elements of $\mathcal{C}$. Accordingly, by Equation~(\ref{symdif}), $\zeta_{G^{\phi}}(C)$ is a product of gains of some cycles in $\mathcal{C}$. 

Now let $\gamma$ be another partial orientation on $G$. If $G^{\phi}$ and $G^{\gamma}$ have the same gain on each cycle in $\mathcal{C}$, then we find that they have the same gain on each fundamental cycle. Then from Theorem~\ref{sw01}, we have $G^{\phi}\thicksim G^{\gamma}$. Conversely, if $G^{\phi}\thicksim G^{\gamma}$ then clearly $G^{\phi}$ and $G^{\gamma}$ have the same gain on each cycle in $\mathcal{C}$. Hence two mixed plane graphs are switching equivalent if and only if they have the same gain on each face cycles.

For all $y\in \{ \pm 1,\pm \mathbf{i} \}^k$, define $\Gamma_G(y)$ to be the set of all $X:=[x_{pq}]_{k \times k}$ satisfying the following conditions:
\begin{enumerate}[label=(\arabic*)]
    \item $x_{pq}= \left\{ \begin{array}{lll}
0 & \textnormal{if $n_{pq}=0$} \\ 1,\mathbf{i} \textnormal{ or } -\mathbf{i} &  \textnormal{if $n_{pq}=1$}\\ 1,-1,\mathbf{i} \textnormal{ or } -\mathbf{i} &  \textnormal{if $n_{pq}>1$;}
\end{array}\right. $
    \item  $\prod\limits_{\substack{q=1\\ n_{pq}\neq 0}}^k x_{pq}=y_p \textnormal{ for } p\in \{1,\ldots,k\}$;
    \item $x_{pq}=\overline{x}_{qp}$ for $p\neq q$.
\end{enumerate}

For a partial orientation $\phi$  on $G$, let $y=(y_1,\ldots, y_k)$, where $y_p=\zeta_{G^{\phi}}(C_{n_p})$ for each $p\in \{1,\ldots,k\}$. Define $T \colon [G^{\phi}] \to \Gamma_G(y)$ such that  $T(G^{\gamma}):=[y_{pq}]$, where 
$$y_{pq}= \left\{ \begin{array}{ll}
\zeta_{G^{\gamma}}(E_{pq})  & \textnormal{ if $n_{pq}\geq 1$} \\ 
0 &  \textnormal{ otherwise.}
\end{array}\right. $$  
We have the following observations. 
\begin{itemize}
\item It is easy to see that each entry of $[y_{pq}]$ satisfies the condition (1) for $\Gamma_G(y)$.
\item Since $G^{\gamma} \in [G^{\phi}]$ and $E(C_{n_p})=E_{p1} \cup E_{p2} \cup \cdots \cup E_{pk}$ is a disjoint union,  we have
$$\prod\limits_{\substack{q=1\\ n_{pq}\neq 0}}^k y_{pq}=\prod\limits_{\substack{q=1\\ n_{pq}\neq 0}}^k \zeta_{G^{\gamma}}(E_{pq})=\zeta_{G^{\gamma}}(C_{n_p}) = \zeta_{G^{\phi}}(C_{n_p}) =y_p$$ 
for each $ p\in \{ 1,\ldots , k\}$. Thus $[y_{pq}]$ satisfies the condition (2) for $\Gamma_G(y)$.
\item Since direction of both the face cycles $C_{n_p}$ and $C_{n_q}$ are clockwise, $y_{pq}=\zeta_{G^{\gamma}}(E_{pq})=\overline{\zeta_{G^{\gamma}}(E_{qp})}=\overline{y}_{qp}$ for $p\neq q$, and so condition (3) holds for $\Gamma_G(y)$. 
\item As $[y_{pq}]$ satisfies all the three conditions for $\Gamma_G(y)$, we have $T(G^{\gamma}) \in \Gamma_G(y)$.

\item For each $Y\in \Gamma_G(y)$, define a mixed graph $G^{\gamma}$ by assigning gains on the edges of $E_{pq}$, according to the clockwise direction of the cycle $C_{n_p}$,  such that $\zeta_{G^{\gamma}}(E_{pq})=y_{pq}$ for $n_{pq}\geq 1$. We see that $T(G^{\gamma})=Y$. Thus $T$ is surjective. 

\item Note that
$$[G^{\phi}]= \bigcup_{Y\in \Gamma_G(y)} T^{-1}(Y), \text{ where } T^{-1}(Y)=\{G^{\gamma}\in [G^{\phi}]\colon T(G^{\gamma})=Y\}.$$
Therefore  $|[G^{\phi}]|= \sum_{Y\in \Gamma_G(y)}| T^{-1}(Y)|$.
\end{itemize}
Consider the following example to explain the preceding discussion.

\noindent\textbf{Example}: Let $G$ be the union of two cycles with a common path, as shown in Figure~\ref{cycle}. Let ${\phi}$ be a partial orientation on $G$.
\begin{figure}[ht]
\centering
\begin{tikzpicture}[scale=0.35]
\node[vertex] (u1) at (12,10) {};
\node [below] at (12.1,11.5) {$u_{q}$};
\node[vertex] (u2) at (15,8) {};
\node [below] at (15.9,8.3) {$u_{1}$};
\node[vertex] (u3) at (15,2) {};
\node [below] at (15.9,2.5) {$u_{p}$};
\node[vertex] (u4) at (12,0) {};
\node [below] at (12.2,-0.2) {$u_{p+1}$};
\node[vertex] (u5) at (9,2) {};
\node [below] at (7.5,2.5) {$u_{p+2}$};
\node[vertex] (u6) at (9,8) {};
\node [below] at (7.5,8.8) {$u_{q-1}$};

\node[vertex] (u7) at (9,4) {};
\node [below] at (7.5,4.5) {$u_{p+3}$};
\node[vertex] (u8) at (9,6) {};
\node [below] at (7.5,6.8) {$u_{q-2}$};
\node[vertex] (u9) at (15,4) {};
\node [below] at (16.4,4.3) {$u_{p-1}$};
\node[vertex] (u10) at (15,6) {};
\node [below] at (15.9,6.3) {$u_{2}$};

\node[vertex] (v1) at (18,10) {};
\node [below] at (18.1,11.5) {$v_{1}$};
\node[vertex] (v2) at (21,8) {};
\node [below] at (21.9,8.7) {$v_{2}$};
\node[vertex] (v3) at (21,2) {};
\node [below] at (22.5,2.5) {$v_{r-1}$};
\node[vertex] (v4) at (18,0) {};
\node [below] at (18.2,-0.2) {$v_{r}$};

\node[vertex] (v5) at (21,4) {};
\node [below] at (22.4,4.7) {$v_{r-2}$};

\node[vertex] (v6) at (21,6) {};
\node [below] at (21.9,6.7) {$v_{3}$};

\node [below] at (12.1,5) {$C_1$};
\node [below] at (18.1,5) {$C_{2}$};

\foreach \from/\to in {u1/u2,u3/u4,u4/u5,u1/u6,v1/u2,v1/v2,v3/v4,v4/u3,u5/u7,u6/u8,u3/u9, u2/u10,v3/v5,v2/v6} \draw (\from) -- (\to);
\draw [dashed] (15,6) -- (15,4);
\draw [dashed] (9,4) -- (9,6);
\draw [dashed] (21,6) -- (21,4);

\end{tikzpicture}
\caption{Union of two cycles with a common path}
\label{cycle}
\end{figure}
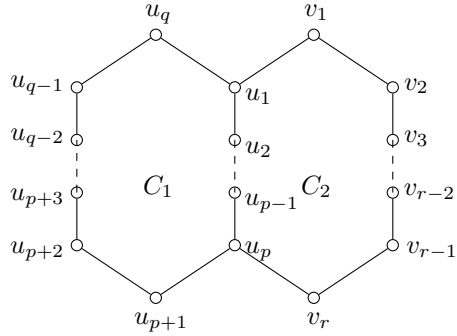 
Let $C_1: = u_1\ldots u_{p-1}u_{p}\ldots u_qu_1$ and $C_2:=v_1v_2\ldots v_ru_pu_{p-1}\ldots u_1v_1$ be the face cycles in $G$, each in clockwise direction. Assume that $n_{pq}\geq 2$ for all $p,q$. It is clear that $E(G)=E_{11}\cup E_{12}\cup E_{22}$.

\begin{enumerate}[label=(\arabic*)]
\item  If $\zeta_{G^{\phi}}(C_1)=1$ and $\zeta_{G^{\phi}}(C_2)=1$, then $y=(1,1)$. We get $\Gamma_G(y)=\{X_1, X_2, X_3,X_4 \} $, where 
$$X_1= \begin{bmatrix} 1&  1\\ 1  & 1 \end{bmatrix},~~X_2= \begin{bmatrix} -1&  -1\\ -1  & -1 \end{bmatrix},~~X_3= \begin{bmatrix} \mathbf{i}& -\mathbf{i}\\ \mathbf{i}  & -\mathbf{i} \end{bmatrix},~~ X_4= \begin{bmatrix} -\mathbf{i}&  \mathbf{i}\\ -\mathbf{i}  & \mathbf{i} \end{bmatrix}.$$ 

For all $G^{\gamma} \in T^{-1}(X_1)$, we have $\zeta_{G^{\gamma}}(E_{11})=1, \zeta_{G^{\gamma}}(E_{12})=1$ and $\zeta_{G^{\gamma}}(E_{22})=1$. Therefore $$|T^{-1}(X_1)|=\alpha_1(n_{11})\alpha_1(n_{12})\alpha_1(n_{22}).$$
Similarly, 
\begin{equation*}
\begin{split}
&|T^{-1}(X_2)|=\alpha_{-1}(n_{11})\alpha_{-1}(n_{12})\alpha_{-1}(n_{22}),\\
&|T^{-1}(X_3)|=\alpha_{\mathbf{i}}(n_{11})\alpha_{-\mathbf{i}}(n_{12})\alpha_{-\mathbf{i}}(n_{22}), \textnormal{ and }\\
&|T^{-1}(X_4)|=\alpha_{-\mathbf{i}}(n_{11})\alpha_{\mathbf{i}}(n_{12})\alpha_{\mathbf{i}}(n_{22}).
\end{split}
\end{equation*} 
Thus 
\begin{equation*}
\begin{split}
|[G^{\phi}]|=& \alpha_1(n_{11})\alpha_1(n_{12})\alpha_1(n_{22})+\alpha_{-1}(n_{11})\alpha_{-1}(n_{12})\alpha_{-1}(n_{22})\\ 
&+\alpha_{\mathbf{i}}(n_{11})\alpha_{-\mathbf{i}}(n_{12})\alpha_{-\mathbf{i}}(n_{22})+\alpha_{-\mathbf{i}}(n_{11})\alpha_{\mathbf{i}}(n_{12})\alpha_{\mathbf{i}}(n_{22}).
\end{split}
\end{equation*}

\item  If $\zeta_{G^{\phi}}(C_1)=1$ and $\zeta_{G^{\phi}}(C_2)=-1$, then $y=(1,-1)$. We get 
$$\Gamma_G(y)=\left\{ \begin{bmatrix} 1&  1\\ 1  & -1 \end{bmatrix}, \begin{bmatrix} -1&  -1\\ -1  & 1 \end{bmatrix},  \begin{bmatrix} \mathbf{i}& -\mathbf{i}\\ \mathbf{i}  & \mathbf{i} \end{bmatrix},  \begin{bmatrix} -\mathbf{i}&  \mathbf{i}\\ -\mathbf{i}  & -\mathbf{i} \end{bmatrix} \right\}.$$ 
Thus 
\begin{equation*}
\begin{split}
|[G^{\phi}]|=& \alpha_1(n_{11})\alpha_1(n_{12})\alpha_{-1}(n_{22})+\alpha_{-1}(n_{11})\alpha_{-1}(n_{12})\alpha_{1}(n_{22})\\
& + \alpha_{\mathbf{i}}(n_{11})\alpha_{-\mathbf{i}}(n_{12})\alpha_{\mathbf{i}}(n_{22})+\alpha_{-\mathbf{i}}(n_{11})\alpha_{\mathbf{i}}(n_{12})\alpha_{-\mathbf{i}}(n_{22}).
\end{split}
\end{equation*} 
\end{enumerate} 
Using similar argument, one can calculate the size of the equivalence classes for the other possible values of $\zeta_{G^{\phi}}(C_1)$ and $\zeta_{G^{\phi}}(C_2)$. In the next result, we introduce a formula to calculate the size of switching equivalence classes of a mixed plane graph.

\begin{theorem}\label{sw31}
Let $G^{\phi}$ be a $2$-connected  mixed plane graph and $C_{n_1},\ldots,C_{n_k}$ be the distinct face cycles in $G$, each considered in clockwise direction. Let  $n_r$ be the order of $C_{n_r}$ for each $r\in \{1,\ldots,k\}$.   Then $$|[G^{\phi}]|=\sum_{X\in \Gamma_G(y)} \prod_{\substack{p,q=1\\ p\leq q,n_{pq}\neq 0}}^k \alpha_{x_{pq}}(n_{pq}),$$ 
where $y=(y_1,\ldots, y_k)$ and $y_p=\zeta_{G^{\phi}}(C_{n_p})$ for each $p\in \{1,\ldots,k\}$.
\end{theorem}
\begin{proof}  Let $X=[x_{pq}]\in \Gamma_G(y)$. Then $T^{-1}(X)$ is the set of all $G^{\gamma}$  such that $\zeta_{G^{\gamma}}(E_{pq})=x_{pq}$ for all $p,q$. Note that $\bigcup_{p,q=1, p\leq q}^{k}E_{pq}$ is a disjoint union of all the edges of $G$. 
Thus 
$$|T^{-1}(X)|= \prod\limits_{\substack{p,q=1\\ p\leq q,n_{pq}\neq 0}}^k \alpha_{x_{pq}}(n_{pq}).$$ 
Now using  $|[G^{\phi}]|= \sum_{X\in \Gamma_G(y)}| T^{-1}(X)|$, we get the desired result.
\end{proof}

\begin{theorem}
Let $G$ be a $2$-connected plane graph. If the number of face cycles in $G$ is $k$, then the number of switching equivalence classes in $\mathcal{M}(G)$ is equal to the size of the set $\{ y\colon y\in \{ \pm 1,\pm \mathbf{i} \}^k , \Gamma_G(y)\neq \emptyset \} $.
\end{theorem}
\begin{proof} Let $W=\{ y\colon y\in \{ \pm 1,\pm \mathbf{i} \}^k , \Gamma_G(y)\neq \emptyset \} $. Let $C_{n_1},\ldots,C_{n_k}$ be the distinct face cycles in $G$, each considered in clockwise direction. Let  $n_r$ be the order of the cycle $C_{n_r}$ for each $r\in \{1,\ldots,k\}$.  We establish a bijective mapping $f$ between $\Omega_{\mathcal{M}}(G)$ and $W$. Define the map $f :\Omega_{\mathcal{M}}(G) \to  W $ such that $f([G^{\phi}])=(\zeta_{G^{\phi}}(C_{n_1}),\ldots,\zeta_{G^{\phi}}(C_{n_k}) )$. Theorem \ref{sw01} gives that $f([G^{\phi}])=f([G^{\gamma}])$ if and only if $[G^{\phi}] = [G^{\gamma}]$. Thus $f$ is well defined and injective. Let $y\in W$, so that $y\in \{ \pm 1,\pm \mathbf{i} \}^k$ and $\Gamma_G(y)\neq \emptyset$. Thus there exists an $X \in \Gamma_G(y)$. Let $X=[x_{pq}]$. Define a mixed graph $G^{\gamma}$ by assigning gains on the edges of $E_{pq}$, according to the clockwise direction of the cycle $C_{n_p}$,  such that $\zeta_{G^{\gamma}}(E_{pq})=x_{pq}$ for $n_{pq}\geq 1$. 
Observe that 
$$\zeta_{G^{\gamma}}(C_{n_p})= \prod\limits_{\substack{q=1\\ E_{pq}\neq \emptyset}}^k \zeta_{G^{\gamma}}(E_{pq})=\prod\limits_{\substack{q=1\\ n_{pq}\neq 0}}^k x_{pq}=y_p \textnormal{ for } p\in \{1,\ldots,k\}.$$ 
Thus $f([G^{\gamma}])=y$, and so $f$ is an onto map. Hence $f$ is a bijective map and the desired result follows.
\end{proof}


\section{Action of automorphism groups on switching equivalence classes}
Consider the gain graphs $G_A^\zeta$ and $G_B^\zeta$ of the matrices $A:=[a_{ij}]$ and $B:=[b_{ij}]$ in $\mathcal{H}_{n}(\Gamma)$, respectively. The gain graphs $G_A^\zeta$ and $G_B^\zeta$ are said to be \emph{isomorphic} if there is a bijective map $f:\{1,\ldots,n\}\rightarrow \{1,\ldots,n\}$ such that $b_{f(i)f(j)}=a_{ij}$ for all $i,j$. 
An \emph{automorphism} of $G_A^\zeta$ is an isomorphism of $G_A^\zeta$ onto itself. The group of all automorphisms of $G_A^\zeta$ is denoted by $\Aut(G_A^\zeta)$. If the automorphisms of $G_A^\zeta$ are identified with their corresponding permutation matrices, then
$$\Aut(G_A^\zeta)= \{P: P\textnormal{ is $n\times n$ permutation matrix and } PA= A P \}.$$ 

The gain graphs $G_A^\zeta$ and $G_B^\zeta$ are said to be \textit{switching isomorphic} if $G_A^\zeta$ is isomorphic to a gain graph $G_C^\zeta$ such that $C$ is switching equivalent to $B$, that is, there is a permutation matrix $P$ and a diagonal matrix  $D(\theta):= \text{diag}[ \theta(1),\ldots,\theta(n)]$ such that $A=(D(\theta)P)^{-1}B(D(\theta)P)$, where $\theta(i)\in \Gamma\setminus \{ 0\}$ for each $i\in \{1,\ldots,n\}$.

Recall that mixed graphs are special types of gain graphs. Thus, isomorphisms and automorphisms of mixed graphs are defined as the same in the gain graphs of their Hermitian adjacency matrices. The group of all automorphisms of a mixed graph $G^{\phi}$ is denoted by $\Aut(G^{\phi})$. The group of all automorphisms of a simple undirected graph $G$ is denoted by $\Aut(G)$. Recall that a partial orientation $\phi$ of an undirected graph $G$ is to specify a  direction according to $\phi$ to each edge in a subset $S$ of $E(G)$. That is, $S$ is the set of all directed edges of $G^{\phi}$. Let $G(S)^{\phi}$ be the subgraph of a mixed graph $G^{\phi}$ with vertex set $V(G)$ and edge set $S$. Note that all the edges of $G(S)^{\phi}$ are directed. Let $G(E\setminus S)$ be the subgraph of $G^{\phi}$ with vertex set $V(G)$ and edge set $E(G)\setminus S$. Note that none of the edges of $G(E\setminus S)$ are directed. It is easy to verify that $H(G^{\phi})=H(G(S)^{\phi})+ H(G(E\setminus S))$. Now we present the following lemma.

\begin{lema} If $G^{\phi}$ is a mixed graph, then $$ \Aut(G^{\phi}) =  \Aut(G) \cap  \Aut (G(S)^{\phi}) =  \Aut(G(S)^{\phi}) \cap  \Aut (G(E \setminus S)).$$   
\end{lema}
\begin{proof}
It is straightforward to check that $\Aut(G^{\phi}) =  \Aut(G) \cap  \Aut (G(S)^{\phi})$. Now we show that $\Aut(G^{\phi}) = \Aut(G(S)^{\phi}) \cap  \Aut (G(E \setminus S)).$ Since $\Aut(G^{\phi}) \subseteq \Aut(G(S)^{\phi})$ and $\Aut(G^{\phi}) \subseteq \Aut (G(E \setminus S))$, we have $\Aut(G^{\phi}) \subseteq \Aut(G(S)^{\phi}) \cap  \Aut (G(E \setminus S))$. Conversely, if $f \in \Aut(G(S)^{\phi})~\cap  \Aut (G(E \setminus S))$ then $f$ is an automorphism of each of $G(S)^{\phi}$ and $G(E \setminus S)$. Thus, $f$ preserves the orientation of the edges of $G^{\phi}$. This gives $f \in \Aut(G^{\phi}) $. Thus $\Aut(G^{\phi}) =  \Aut(G(S)^{\phi}) \cap  \Aut (G(E \setminus S))$.
\end{proof}

For an $A\in \mathcal{H}_{n}(\Gamma)$, let $\mathcal{M}(A)=\{B\in \mathcal{H}_{n}(\Gamma)\colon G_A=G_B\}$. For $B\in \mathcal{M}(A)$, let
\[ [B]=\{X\in \mathcal{M}(A)\colon X \thicksim B\} ~\text{ and } \Omega_{\mathcal{M}}(A)=\{[B]\colon B\in \mathcal{M}(A)\}.\] 
We now define an action of the automorphism group $\Aut(G_A)$ on the set $\Omega_{\mathcal{M}}(A)$, the set of all equivalence classes in $\mathcal{M}(A)$. Let $A\in \mathcal{H}_{n}(\Gamma)$ with $G:=G_A$. For an $f\in \Aut(G)$ and $B\in \mathcal{M}(A)$, define the matrix $f(B)=[b_{uv}^f]$ by $b_{uv}^f=b_{f(u)f(v)}$ for all $u,v$.
 
To verify that this action is well defined, it is enough to show that if $B$ and $C$ are switching equivalent then $f(B)$ and $f(C)$ are also switching equivalent. Assume that $B:=[b_{uv}]$ and $C:=[c_{uv}]$ are switching equivalent. Then there exists a diagonal matrix $D(\theta):= \text{diag}[ \theta(1),\ldots,\theta(n)]$ such that  $B=D(\theta)^{-1} C D(\theta)$, where $\theta(i)\in \Gamma\setminus \{ 0\}$ for each $i\in \{1,\ldots,n\}$. That is, $b_{uv}=\overline{\theta(u)} c_{uv} \theta(v)$ for all $u,v$. 

Consider the diagonal matrix $D(\theta f):= \text{diag}[ \theta(f(1)),\ldots,\theta(f(n))]$. We have
\begin{align*}
\overline{\theta(f(u))}c_{uv}^f \theta (f(v))= \overline{\theta(f(u))}c_{f(u)f(v)} \theta (f(v)) = b_{f(u)f(v)}=b_{uv}^f.
\end{align*}
This means that $f(B)=D(\theta f)^{-1}f(C)D(\theta f)$, that is, $f(B)$ and $f(C)$ are switching equivalent.

\begin{theorem}If $A,B\in \mathcal{H}_{n}(\Gamma)$, then the gain graphs $G_A^\zeta$ and $G_B^\zeta$ are switching isomorphic if and only if $[A]$ and $[B]$ belong to the same orbit of $\Omega_{\mathcal{M}}(A)$ under the action of $\Aut(G_A)$.
\end{theorem}
\begin{proof}
Let $A:=[a_{uv}]$ and $B:=[b_{uv}]$. Let the gain graphs $G_A^\zeta$ and $G_B^\zeta$ be switching isomorphic. Then there exists a gain graph $G_C^\zeta$ such that $G_A^\zeta$ is isomorphic to  $G_C^\zeta$, and $C$ is switching equivalent to $B$. Let $C:=[c_{uv}]$. As $G_A^\zeta$ is isomorphic to  $G_C^\zeta$, there exists $f\in \Aut(G_A)$ such that $c_{uv}=a_{f(u)f(v)}$ for all $u,v$. This means that $C=f(A)$. As $C$ is switching equivalent to $B$, we find that $f(A)$ is switching equivalent to $B$, that is, $[f(A)]=[B]$. Hence $[A]$ and $[B]$ belong to the same orbit of $\Omega_{\mathcal{M}}(A)$ under the action of $\Aut(G_A)$.

Conversely, assume that $[A]$ and $[B]$ belong to the same orbit of $\Omega_{\mathcal{M}}(A)$ under the action of $\Aut(G_A)$. Then there exists $f\in \Aut(G_A)$ such that $[f(A)]=[B]$. We see that $G_A^\zeta$ is isomorphic to $G_{f(A)}^\zeta$ and $f(A)$ is switching equivalent to $B$. Hence $G_A^\zeta$ and $G_B^\zeta$ are switching isomorphic.
\end{proof}

\noindent\textbf{Acknowledgments}

We thank Dr. Debajit Kalita for going through the manuscript and giving some useful comments. We also sincerely thank the anonymous reviewers for carefully reading our manuscript and providing
insightful comments and suggestions that helped us in improving the presentation of the article.


\end{document}